\documentclass[12pt]{amsart}

\usepackage{amssymb,amsmath,graphicx,color,textcomp, amsthm,bbm,bbold, enumerate,booktabs}
\usepackage{chngcntr}
\usepackage{apptools}
\usepackage{mathrsfs}
\usepackage{tikz}
\usetikzlibrary{trees}

\usepackage[margin=1in]{geometry}
\usepackage{tikz-qtree}
\usetikzlibrary{shadows,trees}
\usetikzlibrary{calc,shapes,arrows,positioning}

\usepackage{lscape}

\AtAppendix{\counterwithin{theorem}{section}}
\definecolor{Red}{cmyk}{0,1,1,0}

\definecolor{verde}{cmyk}{1,0,1,0}

\definecolor{loka}{cmyk}{.5,0,1,.5}
\definecolor{azul}{cmyk}{1,1,0,0}

\newcommand{\norm}[1]{\left\lVert#1\right\rVert}

\numberwithin{equation}{section}

\newcommand{\p}{\psi}

\newcommand{\be}{\begin{equation}}
\newcommand{\ee}{\end{equation}}

\newtheorem{theorem}{Theorem}

\newtheorem{definition}{Definition}
\newtheorem{lemma}{Lemma}
\newtheorem{property}{Property}

\begin{document}
\title{ Hilfer-Katugampola fractional derivative}
\author{D. S. Oliveira$^1$}
\address{$^1$ Department of Applied Mathematics, Institute of Mathematics,
 Statistics and Scientific Computation, University of Campinas --
UNICAMP, rua S\'ergio Buarque de Holanda 651,
13083--859, Campinas SP, Brazil\newline
e-mail: {\itshape \texttt{daniela@ime.unicamp.br, capelas@ime.unicamp.br }}}

\author{E. Capelas de Oliveira$^1$}

\begin{abstract} We propose a new fractional derivative, the Hilfer-Katugampola fractional
	derivative. Motivated by the Hilfer derivative this formulation
	interpolates the well-known fractional derivatives of  Hilfer,
	Hilfer-Hadamard, Riemann-Liouville, Hadamard, Caputo, Caputo-Hadamard,
	Liouville, Weyl, generalized and Caputo-type.  As an application, we
	consider a nonlinear fractional differential equation with an initial
	condition using this new formulation. We show that this equation is
	equivalent to a Volterra integral equation and demonstrate the
	existence and uniqueness of solution to the nonlinear initial value
	problem.

\end{abstract}
\maketitle

\section{Introduction}

The number of operators of fractional integration and differentiation has been
increasing during the last years \cite{ ecotenreiro, bhairat, rubeco, Garra, 
katugampola2, katugampola1, katugampola, kilbas, danielaeco, vanterler, capelas}. In order to 
solve fractional differential equations, we mention the following works \cite{adjabi, almeida, 
Furati, Tatar, Barbara}, where the authors propose and prove the equivalence between an initial 
value problem and the Volterra integral equation. Moreover, the existence and uniqueness theorems 
whose solution lies in a convenient space are demonstrated. 

Motivated by some of these formulations and results, we introduce a new
fractional derivative, which resembles the Hilfer and Hilfer-Hadamard fractional 
derivatives \cite{Hilfer, Kassim}. This new fractional derivatives interpolates the
Hilfer, Hilfer-Hadamard, Riemann-Liouville, Hadamard, Caputo, Caputo-Hadamard, 
generalized and Caputo-type fractional derivatives, as well as the Weyl and Liouville 
fractional derivatives for particular cases of integration extremes. More details on 
this fractional derivatives mentioned above can be found in \cite{rubeco, gambo, Hilfer, 
R.Hilfer, katugampola1, katugampola, kilbas, danielaeco}. 

The fact that there exist more than one definition for fractional
derivatives makes choosing the convenient approach a crucial issue in solving a
given problem. Thus, the Hilfer-Katugampola fractional derivative proposed in this
paper is a generalization of the classical fractional derivatives, in order to
overcome the problem of choosing the operator of fractional
differentiation.

In order to obtain the analytical solution of some differential equations
involving this new operator of fractional differentiation, we firstly prove the
equivalence of a nonlinear differential equation, with adequate initial
condition, to a Volterra integral equation. In the sequence we present a
theorem on the existence and uniqueness of solution for such nonlinear
differential equation and the space in which this solution exists.

The text is organized as follows: in section 2, we present results that will be
used in the remaining sections. In section 3, we define our derivative of
non-integer order, the Hilfer-Katugampola fractional derivative, together with
some of its properties.  In section 4, we discuss the equivalence between an
initial value problem and a Volterra integral equation. In section 5, we
present and prove the existence and uniqueness theorem for the initial value
problem presented in the previous section. In section 6 we discuss, 
using the method of successive approximations, the analytical solution
of some fractional differential equations involving this differentiation
operator.

\section{Preliminaries}\label{sec:2}

Before we introduce Hilfer-Katugampola fractional derivatives, we present in
this section some results which will be useful throughout the text, e.g.\ the
Mittag-Leffler functions, the fixed point theorem and generalized fractional
integrals and derivatives.  Though Mittag-Leffler functions are well defined
for complex parameters, in this paper we use only real parameters
\cite{kilbas}.

\begin{definition} \label{Mittag}
A two-parameter Mittag-Leffler function $E_{\alpha, \beta}(x)$,
$\alpha,\beta, x \in\mathbb{R}$ with $\alpha>0$ and $\beta>0$, is defined by
\begin{eqnarray}
E_{\alpha, \beta}(x)=\sum_{k=0}^{\infty}\frac{x^k}{\Gamma(\alpha{k}+\beta)}. \label{MT2}
\end{eqnarray}
If $\beta=1$, we have the one-parameter Mittag-Leffler function, given by 
the following series:
\begin{eqnarray}
E_{\alpha}(x)=\sum_{k=0}^{\infty}\frac{x^k}{\Gamma(\alpha{k}+1)}. \label{MT1}
\end{eqnarray}
\end{definition}
The generalized Mittag-Leffler function $E_{\alpha,l,m}(x)$, introduced
by Kilbas and Saigo \cite{Saigo}, is defined as follows:
\begin{definition} \label{Mittag-generalized}
Let $\alpha,l,m,x\in\mathbb{R}$ such that $\alpha>0$, $m>0$ and $\alpha(jm+l)\notin\mathbb{Z}^{-}$.
The generalized Mittag-Leffler function is defined by
\begin{eqnarray}
E_{\alpha,l,m}(x)=\sum_{k=0}^{\infty}c_{k}x^{k},
\end{eqnarray}
where
\begin{eqnarray}
c_{0}=1 \quad\mbox{and} \quad c_{k}=\prod_{j=0}^{k-1}\frac{\Gamma[\alpha(jm+l)+1]}
{\Gamma[\alpha(jm+l+1)+1]}, \quad k\in\mathbb{N}.
\end{eqnarray}
\end{definition}
\begin{definition} \label{espacos}
Let $\Omega=[a,b]\, (0<a<b<\infty)$ be a finite interval on the half-axis 
$\mathbb{R}^{+}$ and the parameters $\rho>0$ and $0\leq{\gamma}<1$. 

\begin{enumerate}
	\item [\textnormal{(1)}]
We denote by $C[a,b]$ the space of continuous functions $g$ on $\Omega$ 
with the norm
\begin{eqnarray*}	
{\lVert g \rVert}_{C}=\max_{x\in\Omega}|g(x)|.
\end{eqnarray*}
	\item[\textnormal{(2)}]
The weighted space $C_{\gamma, \rho}[a,b]$ of functions $g$ on $(a,b]$ is defined by
\begin{eqnarray}
C_{\gamma,\rho}[a,b]=\left\{g:(a,b]\rightarrow\mathbb{R}:\left(\frac{x^{\rho}-{a}^{\rho}}
{\rho}\right)^{\gamma}g(x)\in{C}[a,b]\right\},
\end{eqnarray}
where $0\leq{\gamma}<1$ and with the norm
\begin{eqnarray}
{\lVert g \rVert}_{C_{\gamma,\rho}}=\norm{\left(\frac{x^{\rho}-{a}^{\rho}}{\rho}\right)^{\gamma}g(x)}_{C}
=\max_{x\in\Omega}\bigg|\left(\frac{x^{\rho}-{a}^{\rho}}{\rho}\right)^{\gamma}g(x)\bigg|, 
\end{eqnarray}
where ${C}_{0,\rho}[a,b]={C}[a,b]$.\\
\item [\textnormal{(3)}] Let $\displaystyle \delta_{\rho}=\left(t^{\rho-1}\frac{d}{dt}\right)$. 
For $n\in\mathbb{N}$ we denote by ${C}_{{\delta}_{{\rho},\gamma}}^{n}[a,b]$ the Banach space
of functions $g$ which are continuously differentiable on $[a,b]$, with operator $\delta_{\rho}$, 
 up to order $(n-1)$ and which have the derivative $\delta_{\rho}^{n}\,g$ of order $n$ on 
$(a,b]$ such that $\delta_{\rho}^{n} g\in C_{\gamma,\rho}[a,b]$, that is,
\begin{eqnarray*}
	{C}_{{\delta}_{{\rho}},\gamma}^{n}[a,b]=\left\{ g:[a,b] \rightarrow \mathbb{R} :
	\delta_{\rho}^{k}\,g\in C[a,b], k=0,1,\ldots,n-1,\delta_{\rho}^{n}\,g\in C_{\gamma,\rho}[a,b]\right\}, 
\end{eqnarray*}
		where $n\in\mathbb{N}$, with the norms 
\begin{eqnarray*}
{\lVert g \rVert}_{{C}_{{\delta}_{{\rho}},\gamma}^{n}}=\sum_{k=0}^{n-1}
{\lVert \delta_{\rho}^{k}\,g \rVert}_{C}+{\lVert \delta_{\rho}^{n}\,g \rVert}_{C_{\gamma},\rho}, 
\quad {\lVert g \rVert}_{{C}_{\delta_{\rho}}^{n}}=\sum_{k=0}^{n}\max_{x\in\Omega}|\delta_{\rho}^{k}\,g(x)|.
\end{eqnarray*}
For $n=0$, we have
\begin{eqnarray*}
{C}_{{\delta}_{{\rho}},\gamma}^{0}[a,b]={C_{{\gamma},\rho}}[a,b].
\end{eqnarray*}
\end{enumerate}
\end{definition}
\begin{lemma}\label{L0}
Let $n\in\mathbb{N}_{0}$ and $\mu_1,\mu_2\in\mathbb{R}$ such that
\begin{eqnarray*}
0\leq{\mu_1}\leq{\mu_2}<1.
\end{eqnarray*}
Then,
\begin{eqnarray*}
C_{\delta_\rho}^{n}[a,b]\longrightarrow{C_{\delta_\rho,\mu_1}^{n}[a,b]}
\longrightarrow{C_{\delta_\rho,\mu_2}^{n}[a,b]},
\end{eqnarray*}
with
\end{lemma}
\begin{eqnarray*}
{\lVert f \rVert}_{{C_{\delta_\rho,\mu_2}^{n}[a,b]}}\leq{K_{\delta_{\rho}}}
{\lVert f \rVert}_{{C_{\delta_\rho,\mu_1}^{n}[a,b]}},
\end{eqnarray*}
where
\begin{eqnarray*}
{K_{\delta_{\rho}}}=\min\left[1,\left(\frac{b^{\rho}-a^{\rho}}{\rho}\right)^{\mu_2-\mu_1}\right], 
\quad \mbox{and} \quad a\neq{0}.
\end{eqnarray*}
In particular,
\begin{eqnarray*}
C[a,b]\longrightarrow{C_{\mu_1,\rho}[a,b]}\longrightarrow{C_{\mu_2,\rho}[a,b]},
\end{eqnarray*}
with
\begin{eqnarray*}
{\lVert f \rVert}_{{C_{\mu_2, \rho}[a,b]}}\leq{\left(\frac{b^{\rho}-a^{\rho}}{\rho}\right)^{\mu_2-\mu_1}}
{\lVert f \rVert}_{C_{\mu_1,\rho}[a,b]}, \quad a\neq{0}.
\end{eqnarray*}
\begin{lemma} \label{L1}
Let $0\leq{\gamma}<1$, $a<c<b$, $g\in{C_{\gamma,\rho}}[a,c]$, $g\in{C}[c,b]$ and $g$ 
continuous at $c$. Then, $g\in{C_{\gamma,\rho}}[a,b]$.
\end{lemma}
\begin{theorem}{\textnormal{\cite{kilbas}}}\label{fixed-point}
Let $(U,d)$ be a nonempty complete metric space; let $0\leq{\omega}<1$, and let $T:U\rightarrow{U}$
be the map such that, for every $u,v\in{U}$, the relation
\begin{eqnarray}
d(Tu,Tv)\leq{\omega}d(u,v), \quad (0\leq{\omega}<1) \label{condicao}
\end{eqnarray}
holds. Then the operator $T$ has a unique fixed point $u^{*}\in{U}$.\\
Futhermore, if $T^{k}\,(k\in\mathbb{N})$ is the sequence of operators defined by
\begin{eqnarray}
T^{1}=T \quad \mbox{and} \quad T^{k}=TT^{k-1}, \quad (k\in\mathbb{N}\backslash\{1\}),
\end{eqnarray}
then, for any $u_0\in{U}$, the sequence $\{T^{k}{u_0}\}_{k=1}^{\infty}$ converges to the
above fixed point $u^*$.
\end{theorem}
The map $T:U\rightarrow{U}$ satisfying condition \textnormal{Eq.(\ref{condicao})}
is called a contractive map.
\begin{definition}{\textnormal{\cite{katugampola1}}}\label{generalized-integral}
Let $\alpha,c\in\mathbb{R}$ with $\alpha>0$ and $\varphi\in{X}_{c}^{p}(a,b)$, 
where ${X}_{c}^{p}(a,b)$ consists of those complex-valued Lebesgue measurable
functions. The generalized fractional integrals, left- and right-sided,
are defined, respectively, by
\begin{eqnarray}
(^{\rho}\mathcal{J}_{a^+}^{\alpha}\varphi)(x)=\frac{\rho^{1-\alpha}}{\Gamma(\alpha)}
\int_{a}^{x}\frac{t^{\rho-1}\,\varphi(t)}{(x^{\rho}-t^{\rho})^{1-\alpha}}dt, \qquad x>a \label{int-gen-esq}
\end{eqnarray}
and
\begin{eqnarray}
(^{\rho}\mathcal{J}_{b^-}^{\alpha}\varphi)(x)=\frac{\rho^{1-\alpha}}{\Gamma(\alpha)}
\int_{x}^{b}\frac{t^{\rho-1}\,\varphi(t)}{(t^{\rho}-x^{\rho})^{1-\alpha}}dt, \qquad x<b \label{int-gen-dir}
\end{eqnarray}
with $\rho>0$.
\end{definition}
Similarly, we define generalized fractional derivatives which correspond to
generalized fractional integrals, \textnormal{Eq.{(\ref{int-gen-esq})}} and 
\textnormal{Eq.{(\ref{int-gen-dir})}}.
\begin{definition}{\textnormal{\cite{katugampola}}}\label{der-gen}
Let $\alpha, \rho \in\mathbb{R}$ such that $\alpha, \rho >{0}$, $\alpha\notin\mathbb{N}$, 
and let $n=[\alpha]+1$, where 
$[\alpha]$ is the integer part of $\alpha$. The generalized 
fractional derivatives, $(^{\rho}\mathcal{D}_{a^+}^{\alpha}\varphi)(x)$ 
and $(^{\rho}\mathcal{D}_{a^-}^{\alpha}\varphi)(x)$, left- and right-sided, 
are defined by
\begin{eqnarray}
(^{\rho}\mathcal{D}_{a^+}^{\alpha}\varphi)(x)&=&\delta_{\rho}^{n}(^{\rho}
\mathcal{J}_{a^+}^{n-\alpha}\varphi)(x)\nonumber\\
&=&\frac{\rho^{1-n+\alpha}}{\Gamma(n-\alpha)}\left(x^{1-\rho}\frac{d}{dx}\right)^{n}
\int_{a}^{x}\frac{t^{\rho-1}\,\varphi(t)}{(x^{\rho}-t^{\rho})^{1-n+\alpha}}dt, \label{der-gen-esq}
\end{eqnarray}
and
\begin{eqnarray}
(^{\rho}\mathcal{D}_{a^-}^{\alpha}\varphi)(x)&=&(-1)^{n}\delta_{\rho}^{n}(^{\rho}
\mathcal{J}_{a^-}^{n-\alpha}\varphi)(x)\nonumber\\
&=&\frac{(-1)^{n}\rho^{1-n+\alpha}}{\Gamma(n-\alpha)}\left(x^{1-\rho}\frac{d}{dx}\right)^{n}
\int_{x}^{b}\frac{t^{\rho-1}\,\varphi(t)}{(x^{\rho}-t^{\rho})^{1-n+\alpha}}dt,
\end{eqnarray}
respectively, if the integrals exist and $\displaystyle {\delta}_{\rho}^{n}=
\left(x^{1-\rho}\frac{d}{dx}\right)^{n}$.
\end{definition}
\begin{theorem}\label{semigrupo}
Let $\alpha>0$, $\beta>0$, $1\leq{p}\leq{\infty}$, $0<a<b<\infty$ and $\rho,c\in\mathbb{R}$,  
$\rho\geq{c}$. Then, for $\varphi\in{X}_{c}^{p}(a,b)$ the semigroup property is valid, i.e.\  

\begin{eqnarray*}
(^{\rho}\mathcal{J}_{a^+}^{\alpha}\,^{\rho}\mathcal{J}_{a^+}^{\beta}\varphi)(x)=
(^{\rho}\mathcal{J}_{a^+}^{\alpha+\beta}\varphi)(x).
\end{eqnarray*}
\begin{proof}
See \textnormal{\cite{katugampola1}}.
\end{proof}
\end{theorem}
\begin{lemma} \label{L2}
Let $x>a$, ${^{\rho}\mathcal{J}_{a^+}^{\alpha}}$ and ${^{\rho}\mathcal{D}_{a^+}^{\alpha}}$,
as defined in \textnormal{Eq.(\ref{int-gen-esq})} and \textnormal{Eq.(\ref{der-gen-esq})},
respectively. Then, for $\alpha\geq{0}$ and $\xi>0$, we have
\begin{eqnarray*}
\left[{^{\rho}\mathcal{J}_{a^+}^{\alpha}}\left(\frac{t^{\rho}-a^{\rho}}{\rho}\right)^{\xi-1}\right](x)&=&\frac{\Gamma(\xi)}{\Gamma(\alpha+\xi)}\left(\frac{x^{\rho}-a^{\rho}}{\rho}\right)^{\alpha+\xi-1},\\
\left[{^{\rho}\mathcal{D}_{a^+}^{\alpha}}\left(\frac{t^{\rho}-a^{\rho}}{\rho}\right)^{\alpha-1}\right](x)&=&0, \quad 0<\alpha<1.
\end{eqnarray*}
\begin{proof}
See \textnormal{\cite{almeida}}.
\end{proof}
\end{lemma}
\begin{lemma}\label{Lem-2}
For $\alpha>0$, ${^{\rho}\mathcal{J}_{a^+}^{\alpha}}$ maps ${C}[a,b]$ into ${C}[a,b]$.
\end{lemma}
\begin{lemma} \label{Lem-3}
Let $\alpha>0$ and $0\leq\gamma<{1}$. Then, ${^{\rho}\mathcal{J}_{a^+}^{\alpha}}$ is
bounded from ${C}_{\gamma,\rho}[a,b]$ into ${C}_{\gamma,\rho}[a,b]$.
\end{lemma}
\begin{lemma}
Let $\alpha>0$ and $0\leq\gamma<{1}$. If $\gamma\leq{\alpha}$, then ${^{\rho}\mathcal{J}_{a^+}^{\alpha}}$
is bounded from ${C}_{\gamma,\rho}[a,b]$ into ${C}[a,b]$.
\end{lemma}
\begin{lemma} \label{L3}
Let $0<a<b<\infty$, $\alpha>0$, $0\leq{\gamma}<{1}$ and $\varphi\in{C_{\gamma,\rho}[a,b]}$. 
If $\alpha>\gamma$, then ${^{\rho}\mathcal{J}_{a^+}^{\alpha}}\varphi$ is continuous on $[a,b]$ and
\begin{eqnarray*}
({^{\rho}\mathcal{J}_{a^+}^{\alpha}}\varphi)(a)=\lim_{x \to a^+}({^{\rho}\mathcal{J}_{a^+}^{\alpha}}\varphi)(x)=0.
\end{eqnarray*}
\begin{proof}
Since $\varphi\in{C_{\gamma,\rho}[a,b]}$, then $\displaystyle\left
(\frac{x^{\rho}-a^{\rho}}{\rho}\right)^{\gamma}\varphi(x)$ is continuous on 
$[a,b]$ and
\begin{eqnarray*}
\bigg|\left(\frac{x^{\rho}-a^{\rho}}{\rho}\right)^{\gamma}\varphi(x)\bigg|\leq{M},
	\quad x\in[a,b] , 
\end{eqnarray*}
for some positive constant $M$. Consequently,
\begin{eqnarray*}
|({^{\rho}\mathcal{J}_{a^+}^{\alpha}}\varphi)(x)|\leq{M}\left[{^{\rho}
\mathcal{J}_{a^+}^{\alpha}}\left(\frac{t^{\rho}-a^{\rho}}{\rho}\right)^{-\gamma}\right](x)
\end{eqnarray*}
and by \textnormal{Lemma \ref{L2}}, we can write
\begin{eqnarray}
|({^{\rho}\mathcal{J}_{a^+}^{\alpha}}\varphi)(x)|\leq{M}\frac{\Gamma(1-\gamma)}
{\Gamma(\alpha-\gamma+1)}\left(\frac{x^{\rho}-a^{\rho}}{\rho}\right)^{\alpha-\gamma}. 
\label{integral}
\end{eqnarray}
As $\alpha>\gamma$, the right-hand side of \textnormal{Eq.(\ref{integral})} 
goes to zero when $x \rightarrow a^+$.
\end{proof}
\end{lemma}
\begin{lemma}
Let $\alpha>0$, $0\leq{\gamma}<1$ and $\varphi\in{C}_{\gamma}[a,b]$. Then,
\begin{eqnarray*}
({^{\rho}\mathcal{D}_{a^+}^{\alpha}}\,{^{\rho}\mathcal{J}_{a^+}^{\alpha}}\varphi)(x)=\varphi(x),
\end{eqnarray*}
for all $x\in(a,b]$.
\begin{proof}
See \textnormal{\cite{katugampola}}.
\end{proof}
\end{lemma}
\begin{lemma}\label{L4}
Let $0<\alpha<1$, $0\leq{\gamma}<1$. If $\varphi\in{C}_{\gamma}[a,b]$ and 
${^{\rho}\mathcal{J}_{a^+}^{1-\alpha}}\varphi\in{C}_{\gamma}^{1}[a,b]$, then
\begin{eqnarray*}
({^{\rho}\mathcal{J}_{a^+}^{\alpha}}\,{^{\rho}\mathcal{D}_{a^+}^{\alpha}}\varphi)(x)=
\varphi(x)-\frac{({^{\rho}\mathcal{J}_{a^+}^{1-\alpha}}\varphi)(a)}{\Gamma(\alpha)}
\left(\frac{x^{\rho}-a^{\rho}}{\rho}\right)^{\alpha-1},
\end{eqnarray*}
for all $x\in(a,b]$.
\begin{proof}
The proof uses integration by parts, with the choice 
$u=(x^{\rho}-t^{\rho})^{\alpha-1}$ and $\displaystyle dv=\frac{d}{dt}
({^{\rho}\mathcal{J}_{a^+}^{1-\alpha}}\varphi)(t)dt$.
\end{proof}
\end{lemma}

\begin{lemma}
Let $\alpha>0$, $0\leq{\gamma}<1$ and $\varphi\in{C}_{\gamma}[a,b]$. Then,
\begin{eqnarray*}
({^{\rho}\mathcal{D}_{a^+}^{\alpha}}\,{^{\rho}\mathcal{J}_{a^+}^{\alpha}}\varphi)(x)=\varphi(x),
\end{eqnarray*}
for all $x\in(a,b]$.
\begin{proof}
See \textnormal{\cite{katugampola}}.
\end{proof}
\end{lemma}
\begin{lemma}\label{L4}
Let $0<\alpha<1$, $0\leq{\gamma}<1$. If $\varphi\in{C}_{\gamma}[a,b]$ and 
${^{\rho}\mathcal{J}_{a^+}^{1-\alpha}}\varphi\in{C}_{\gamma}^{1}[a,b]$, then
\begin{eqnarray*}
({^{\rho}\mathcal{J}_{a^+}^{\alpha}}\,{^{\rho}\mathcal{D}_{a^+}^{\alpha}}\varphi)(x)=
\varphi(x)-\frac{({^{\rho}\mathcal{J}_{a^+}^{1-\alpha}}\varphi)(a)}{\Gamma(\alpha)}
\left(\frac{x^{\rho}-a^{\rho}}{\rho}\right)^{\alpha-1},
\end{eqnarray*}
for all $x\in(a,b]$.
\begin{proof}
The proof uses integration by parts, with the choice 
$u=(x^{\rho}-t^{\rho})^{\alpha-1}$ and $\displaystyle dv=\frac{d}{dt}
({^{\rho}\mathcal{J}_{a^+}^{1-\alpha}}\varphi)(t)dt$.
\end{proof}
\end{lemma}


\section{Hilfer-Katugampola fractional derivative}\label{sec:3}

In this section, our main result, we introduce the Hilfer-Katugampola
fractional derivative and discuss other formulations for fractional
derivatives.

\begin{definition}\label{Hilfer-Katugampola}
Let order $\alpha$ and type $\beta$ satisfy $n-1<\alpha\leq{n}$ and $0\leq{\beta}\leq{1}$,
with $n\in\mathbb{N}$. 
The fractional derivative (left-sided/right-sided), with respect to $x$, with 
$\rho>0$ of a function $\varphi\in{C_{{1-\gamma},\rho}}[a,b]$, is defined by
\begin{eqnarray*}
({^{\rho}\mathcal{D}^{\alpha,\beta}_{a\pm}}\varphi)(x)&=&\left(\pm\,{^{\rho}\mathcal{J}_{a\pm}^{\beta(n-\alpha)}}\left(t^{\rho-1}\frac{d}{dt}\right)^{n}{^{\rho}\mathcal{J}_{a\pm}^{(1-\beta)(n-\alpha)}}\varphi\right)(x)\\
&=&\left(\pm\,{^{\rho}\mathcal{J}_{a\pm}^{\beta(n-\alpha)}}\delta_{\rho}^{n}\,{^{\rho}\mathcal{J}_{a\pm}^{(1-\beta)(n-\alpha)}}\varphi\right)(x),
\end{eqnarray*}
where $\mathcal{J}$ is the generalized fractional integral 
given in \textnormal{Definition \ref{generalized-integral}}.
In this paper we consider the case $n=1$ only, because the Hilfer derivative and
the Hilfer-Hadamard derivative are discussed with $0<\alpha<1$.
\end{definition}
Note that we present and discuss our new results involving the Hilfer-Katugampola
fractional derivative using only the left-sided operator. 
An analogous procedure can be developed using the right-sided operator. The following
property shows that it is possible to write operator
${^{\rho}\mathcal{D}^{\alpha,\beta}_{a^+}}$
in terms of the operator given in Definition \ref{der-gen}.
\begin{property} \label{P1}
The operator ${^{\rho}\mathcal{D}^{\alpha,\beta}_{a^+}}$ can be written as
\begin{eqnarray*}
{^{\rho}\mathcal{D}^{\alpha,\beta}_{a^+}}={^{\rho}\mathcal{J}_{a^+}^{\beta(1-\alpha)}}
\delta_{\rho}\,{^{\rho}\mathcal{J}_{a^+}^{1-\gamma}}={^{\rho}\mathcal{J}_{a^+}^{\beta(1-\alpha)}}
\,{^{\rho}\mathcal{D}_{a^+}^{\gamma}}, \quad \gamma=\alpha+\beta(1-\alpha).
\end{eqnarray*}
\begin{proof}
From definition of the generalized fractional integral, we have
\begin{footnotesize}
\begin{eqnarray*}
({^{\rho}\mathcal{D}^{\alpha,\beta}_{a^+}}\varphi)(x)&=&{^{\rho}\mathcal{J}_{a^+}^{\beta(1-\alpha)}}
\left(x^{1-\rho}\frac{d}{dx}\right)\\
&\times&\left\{\frac{\rho^{1-(1-\beta})(1-\alpha)}{\Gamma((1-\beta})
(1-\alpha))\int_{a}^{x}\frac{t^{\rho-1}}{(x^{\rho}-t^{\rho})^{1-(1-\beta})(1-\alpha)}\varphi(t)dt\right\}\\
&=&\left[{^{\rho}\mathcal{J}_{a^+}^{\beta(1-\alpha)}}\frac{\rho^{1+\alpha+\beta}-\alpha\beta}
{\Gamma((1-\beta})(1-\alpha)-1)\right.\\
&\times& \left.\int_{a}^{x}\frac{t^{\rho-1}}{(x^{\rho}-t^{\rho})^{1+\alpha+\beta}-\alpha\beta}\varphi(t)dt
\right](x)\\
&=&({^{\rho}\mathcal{J}_{a^+}^{\beta(1-\alpha)}}\,{^{\rho}\mathcal{D}_{a^+}^{\gamma}}\varphi)(x),
\end{eqnarray*}
\end{footnotesize}
where operator $\mathcal{D}$ is the generalized fractional derivative 
	given in \textnormal{Definition \ref{der-gen}}.
\end{proof}
\end{property}
\begin{property} \label{Hilfer-Katugampola-parameters}

The fractional derivative ${^{\rho}\mathcal{D}^{\alpha,\beta}_{a^+}}$ is an
	interpolator of the following fractional derivatives: Hilfer
	$(\rho\rightarrow{1})$ \textnormal{\cite{Hilfer}}, Hilfer-Hadamard
	$(\rho\rightarrow{0^+})$ \textnormal{\cite{Kassim}}, generalized
	$(\beta=0)$ \textnormal{\cite{katugampola1}}, Caputo-type $(\beta=1)$
	\textnormal{\cite{danielaeco}}, Riemann-Liouville $(\beta=0,
	\rho\rightarrow{1})$ \textnormal{\cite{kilbas}}, Hadamard $(\beta=0,
	\rho\rightarrow{0^+})$ \textnormal{\cite{kilbas}}, Caputo $(\beta=1,
	\rho\rightarrow{1})$ \textnormal{\cite{kilbas}}, Caputo-Hadamard
	$(\beta=1, \rho\rightarrow{0^+})$ \textnormal{\cite{gambo}}, Liouville
	$(\beta=0, \rho\rightarrow{1}, a=0)$ \textnormal{\cite{kilbas}} and
	Weyl $(\beta=0, \rho\rightarrow{1}, a=-\infty)$
	\textnormal{\cite{R.Hilfer}}.  This fact is illustrated in the diagram
	below.

\end{property}
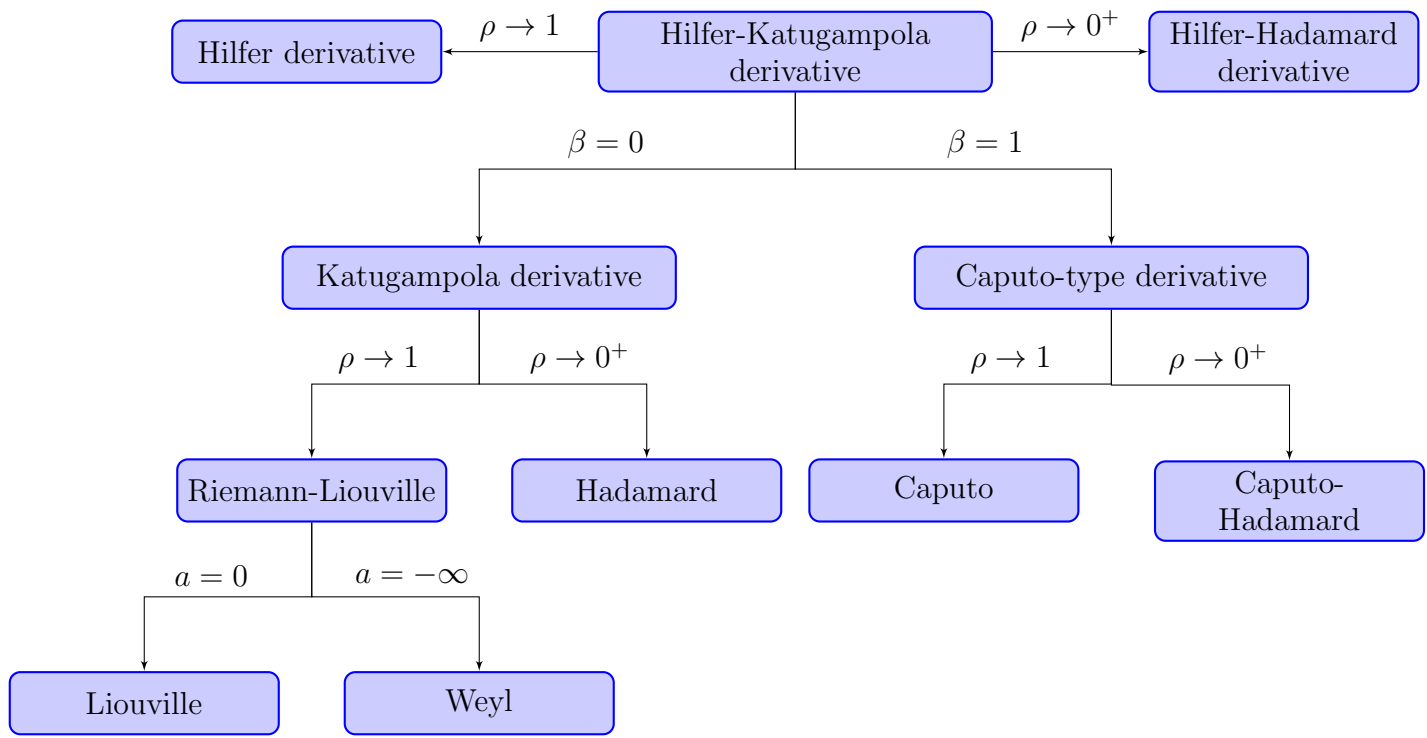
\begin{landscape}

\centering

\tikzstyle{block} = [rectangle, draw=blue, thick, fill=blue!20, align=center,
text width=12em, node distance=1cm, text centered, rounded corners,  minimum width=3cm, minimum height=2em]

\tikzstyle{block2} = 
    [rectangle, draw=blue, thick, fill=blue!20, align=center,
text width=8em, node distance=2cm, text centered, rounded corners,  minimum width=3cm, minimum height=2em]

\tikzstyle{block3} = 
    [rectangle, draw=blue, thick, fill=blue!20, align=center,
text width=8em, node distance=2cm, text centered, rounded corners,  minimum width=3cm, minimum height=2em]


\tikzstyle{line} = 
    [draw, -latex']

\begin{tikzpicture}[node distance = 1cm, auto]

    \node [block] (Hilfer) {Hilfer-Katugampola derivative};
    \node [block, below left of= Hilfer, node distance=4.25cm,  xshift=-1.2cm] (Katugampola) {Katugampola derivative};
    \node [block, below right of= Hilfer, node distance=4.25cm,  xshift=1.2cm] (Caputo-type) {Caputo-type derivative};

    \node [block2, below left of=Katugampola, node distance=4cm,  xshift=0.6cm] (RL) {Riemann-Liouville};
    \node [block2, below right of=Katugampola, node distance=4cm,  xshift=-0.6cm] (Hadamard) {Hadamard};
    \node [block2, below left of=Caputo-type, node distance=4cm,  xshift=0.6cm] (Caputo) {Caputo};
    \node [block2, below right of=Caputo-type, node distance=4.2cm,  xshift=-0.6cm] (Caputo-Hadmard) {Caputo-Hadamard};

    \node [block3, below left of=RL, node distance=4cm,  xshift=0.6cm] (Liouville) {Liouville };
    \node [block3, below right of=RL, node distance=4cm,  xshift=-0.6cm] (Weyl) {Weyl};
    
    \node [block2, right of=Hilfer, node distance=5cm,  xshift=1.5cm] (H-H) {Hilfer-Hadamard derivative} ;
		
    \node [block2, left of=Hilfer, node distance=5cm,  xshift=-1.5cm] (H) {Hilfer derivative} ;

    \path[line] let \p1=(Hilfer.south), \p2=(Katugampola.north) in (Hilfer.south) --  +(0,0.5*\y2-0.5*\y1) -| node [pos=0.3, above] {$\beta=0$} (Katugampola.north);
    \path[line] let \p1=(Hilfer.south), \p2=(Caputo-type.north) in (Hilfer.south) -- +(0,0.5*\y2-0.5*\y1) -| node [pos=0.3, above] {$\beta=1$} (Caputo-type.north);

    \path[line] let \p1=(Katugampola.south), \p2=(RL.north) in (Katugampola.south) -- +(0,0.5*\y2-0.5*\y1) -| node [pos=0.3, above] {$\rho\rightarrow{1}$} (RL.north);
    \path[line] let \p1=(Katugampola.south), \p2=(Hadamard.north) in (Katugampola.south) -- +(0,0.5*\y2-0.5*\y1) -| node [pos=0.3, above] {$\rho\rightarrow{0^+}$} (Hadamard.north);

    \path[line] let \p1=(Caputo-type.south), \p2=(Caputo.north) in (Caputo-type.south) -- +(0,0.5*\y2-0.5*\y1) -| node [pos=0.3, above] {$\rho\rightarrow{1}$} (Caputo.north);
    \path[line] let \p1=(Caputo-type.south), \p2=(Caputo-Hadmard.north) in (Caputo-type.south) -- +(0,0.5*\y2-0.5*\y1) -| node [pos=0.3, above] {$\rho\rightarrow{0^+}$} (Caputo-Hadmard.north);
    
     \path[line] let \p1=(RL.south), \p2=(Liouville.north) in (RL.south) -- +(0,0.5*\y2-0.5*\y1) -| node [pos=0.3, above] {$a=0$} (Liouville.north);
    \path[line] let \p1=(RL.south), \p2=(Weyl.north) in (RL.south) -- +(0,0.5*\y2-0.5*\y1) -| node [pos=0.3, above] {$a=-\infty$} (Weyl.north);
		
	 \path [line] (Hilfer) -- node {$\rho\rightarrow{0^+}$}(H-H) ;	

	 \path [line] (Hilfer) -- node[above] {$\rho\rightarrow{1}$}(H) ;

\end{tikzpicture}


\end{landscape}
\begin{property}\label{espacos2}
We consider the following parameters $\alpha,\beta,\gamma,\mu$ satisfying
\begin{eqnarray*}
\gamma =\alpha+\beta(1-\alpha), \quad\quad 0<\alpha, \beta, \gamma <1, \quad\quad 0 \leq \mu <1.
\end{eqnarray*}
Thus, we define the spaces 
\begin{eqnarray*}
{C}_{1-\gamma, \mu}^{\alpha,\beta}[a,b]=\{\varphi\in{C}_{1-\gamma,\rho}[a,b], 
{^{\rho}\mathcal{D}^{\alpha,\beta}_{a^+}}\varphi\in{C}_{\mu,\rho}[a,b]\},
\end{eqnarray*}
and
\begin{eqnarray*}
{C}_{1-\gamma,\rho}^{\gamma}[a,b]=\{\varphi\in{C}_{1-\gamma,\rho}[a,b], {^{\rho}
\mathcal{D}^{\gamma}_{a^+}}\varphi\in{C}_{1-\gamma,\rho}[a,b]\},
\end{eqnarray*}
where ${C}_{\mu,\rho}[a,b]$ and ${C}_{1-\gamma,\rho}[a,b]$ are weighted spaces of 
	continuous functions on $(a,b]$ defined by \rm{item (2)}
	in \textnormal{Definition \ref{espacos}}.
Since ${^{\rho}\mathcal{D}^{\alpha,\beta}_{a^+}}\varphi=
	{{^{\rho}\mathcal{J}^{\gamma(1-\alpha)}_{a^+}}\,
	{^{\rho}\mathcal{D}^{\gamma}_{a^+}}\varphi}$, it follows from \textnormal{Lemma \ref{Lem-3}} that
\begin{eqnarray*}
{C}_{1-\gamma}^{\gamma}[a,b]\subset{C}_{1-\gamma}^{\alpha,\beta}[a,b].
\end{eqnarray*}
\end{property}
\begin{lemma} \label{L5}
Let $0<\alpha<1$, $0\leq\beta\leq{1}$ and $\gamma=\alpha+\beta(1-\alpha)$. 
If $\varphi\in{C}_{1-\gamma}^{\gamma}[a,b]$, then
\begin{eqnarray}
{^{\rho}\mathcal{J}_{a^+}^{\gamma}}{^{\rho}\mathcal{D}^{\gamma}_{a^+}}\varphi=
{^{\rho}\mathcal{J}_{a^+}^{\alpha}}{^{\rho}\mathcal{D}^{\alpha,\beta}_{a^+}}\varphi 
\label{integral-derivada}
\end{eqnarray}
and
\begin{eqnarray}
{^{\rho}\mathcal{D}^{\gamma}_{a^+}}{^{\rho}\mathcal{J}_{a^+}^{\alpha}}\varphi=
{^{\rho}\mathcal{D}^{\beta(1-\alpha)}_{a^+}}\varphi. \label{derivada-integral}
\end{eqnarray}
\begin{proof}
We first prove \textnormal{Eq.(\ref{integral-derivada})}. 
Using \textnormal{Theorem \ref{semigrupo}} and \textnormal{Property \ref{P1}}, 
we can write
\begin{eqnarray*}
{^{\rho}\mathcal{J}_{a^+}^{\gamma}}{^{\rho}\mathcal{D}^{\gamma}_{a^+}}\varphi&=&
{^{\rho}\mathcal{J}_{a^+}^{\gamma}}{^{\rho}\mathcal{J}_{a^+}^{-\beta(1-\alpha)}}
{^{\rho}\mathcal{D}^{\alpha,\beta}_{a^+}}\varphi\\
&=&{^{\rho}\mathcal{J}_{a^+}^{\alpha+\beta-\alpha\beta}}{^{\rho}\mathcal{J}_{a^+}^{-\beta+\alpha\beta}}
{^{\rho}\mathcal{D}^{\alpha,\beta}_{a^+}}\varphi
={^{\rho}\mathcal{J}_{a^+}^{\alpha}}{^{\rho}\mathcal{D}^{\alpha,\beta}_{a^+}}\varphi.
\end{eqnarray*}
To prove \textnormal{Eq.(\ref{derivada-integral})}, we use 
\textnormal{Definition \ref{Hilfer-Katugampola}} and
\textnormal{Theorem \ref{semigrupo}} to get
\begin{eqnarray*}
{^{\rho}\mathcal{D}^{\gamma}_{a^+}}{^{\rho}\mathcal{J}_{a^+}^{\alpha}}\varphi&=&
\delta_{\rho}\,{^{\rho}\mathcal{J}_{a^+}^{1-\gamma}}\,{^{\rho}\mathcal{J}_{a^+}^{\alpha}}\varphi
=\delta_{\rho}\,{^{\rho}\mathcal{J}_{a^+}^{1-\beta+\alpha\beta}}\varphi\\
&=&\delta_{\rho}\,{^{\rho}\mathcal{J}_{a^+}^{1-\beta(1-\alpha)}}\varphi
={^{\rho}\mathcal{D}^{\beta(1-\alpha)}_{a^+}}\varphi.
\end{eqnarray*}
\end{proof}
\end{lemma}
\begin{lemma}\label{L6}
Let $\varphi\in{L}^{1}(a,b)$. If ${^{\rho}\mathcal{D}^{\beta(1-\alpha)}_{a^+}}\varphi$ 
exists on ${L}^{1}(a,b)$, then
\begin{eqnarray*}
{^{\rho}\mathcal{D}^{\alpha,\beta}_{a^+}}{^{\rho}\mathcal{J}_{a^+}^{\alpha}}\varphi={^{\rho}\mathcal{J}_{a^+}^{\beta(1-\alpha)}}{^{\rho}\mathcal{D}^{\beta(1-\alpha)}_{a^+}}\varphi.
\end{eqnarray*}
\begin{proof}
From \textnormal{Lemma \ref{L2}}, \textnormal{Definition \ref{der-gen}} and 
\textnormal{Definition \ref{Hilfer-Katugampola}}, we obtain 
\begin{eqnarray*}
{^{\rho}\mathcal{D}^{\alpha,\beta}_{a^+}}{^{\rho}\mathcal{J}_{a^+}^{\alpha}}\varphi
&=&{^{\rho}\mathcal{J}_{a^+}^{\beta(1-\alpha)}}\,{^{\rho}\mathcal{D}^{\gamma}_{a^+}}
{^{\rho}\mathcal{J}_{a^+}^{\alpha}}\varphi
={^{\rho}\mathcal{J}_{a^+}^{\beta(1-\alpha)}}\,\delta_{\rho}\,{^{\rho}\mathcal{J}_{a^+}^{1-\gamma}}\,
{^{\rho}\mathcal{J}_{a^+}^{\alpha}}\varphi\\
&=&{^{\rho}\mathcal{J}_{a^+}^{\beta(1-\alpha)}}\,\delta_{\rho}\,
{^{\rho}\mathcal{J}_{a^+}^{1-\beta(1-\alpha)}}\varphi
={^{\rho}\mathcal{J}_{a^+}^{\beta(1-\alpha)}}{^{\rho}\mathcal{D}^{\beta(1-\alpha)}_{a^+}}\varphi.
\end{eqnarray*}
\end{proof}
\end{lemma}
\begin{lemma}
Let $0<\alpha<1$, $0\leq{\beta}\leq{1}$ and $\gamma=\alpha+\beta(1-\alpha)$. 
If $\varphi\in{C}_{1-\gamma}[a,b]$ and ${^{\rho}\mathcal{J}_{a^+}^{1-\beta(1-\alpha)}}
\in{C}^{1}_{1-\gamma}[a,b]$, then ${^{\rho}\mathcal{D}^{\alpha,\beta}_{a^+}}\,
{^{\rho}\mathcal{J}_{a^+}^{\alpha}}\varphi$ exists on $(a,b]$ and
\begin{eqnarray}
{^{\rho}\mathcal{D}^{\alpha,\beta}_{a^+}}\,{^{\rho}\mathcal{J}_{a^+}^{\alpha}}\varphi=\varphi, 
\quad x\in(a,b].
\end{eqnarray}
\begin{proof}
Using \textnormal{Lemma \ref{L4}}, \textnormal{Lemma \ref{L2}} and \textnormal{Lemma \ref{L6}}, 
we obtain
\begin{eqnarray*}
({^{\rho}\mathcal{D}^{\alpha,\beta}_{a^+}}\,{^{\rho}\mathcal{J}_{a^+}^{\alpha}}\varphi)(x)
&=&({^{\rho}\mathcal{J}_{a^+}^{\beta(1-\alpha)}}\,{^{\rho}\mathcal{D}^{\beta(1-\alpha)}_{a^+}}\varphi)(x)\\
&=&\varphi(x)-\frac{({^{\rho}\mathcal{J}_{a^+}^{\beta(1-\alpha)}}\varphi)(a)}{\Gamma(\alpha)}
\left(\frac{x^{\rho}-a^{\rho}}{\rho}\right)^{\beta(1-\alpha)-1}\\
&=&\varphi(x), \quad x\in(a,b].
\end{eqnarray*}
\end{proof}
\end{lemma}


\section{Equivalence between the generalized Cauchy problem and the Volterra integral equation}

We consider the following nonlinear fractional differential equation
\begin{eqnarray}
({^{\rho}\mathcal{D}^{\alpha,\beta}_{a^+}}\varphi)(x)=f(x,\varphi(x)), \quad x>a>0 \label{Cauchy}
\end{eqnarray}
where $0<\alpha<1$, $0\leq{\beta}\leq{1}$, $\rho>0$, with the initial condition
\begin{eqnarray}
({^{\rho}\mathcal{J}_{a^+}^{1-\gamma}}\varphi)(a)=c, \quad \textnormal{with} 
\quad \gamma=\alpha+\beta(1-\alpha), \quad c\in\mathbb{R}. \label{cond-inicial}
\end{eqnarray}
The following theorem yields the equivalence between the problem Eq.(\ref{Cauchy})-Eq.(\ref{cond-inicial}) 
and the Volterra integral equation, given by
\begin{eqnarray}
\varphi(x)&=&\frac{c}{\Gamma(\gamma)}\left(\frac{x^{\rho}-a^{\rho}}{\rho}\right)^{\gamma-1} 
+\frac{1}{\Gamma(\alpha)}\int_{a}^{x}\left(\frac{x^{\rho}-t^{\rho}}{\rho}\right)^{\alpha-1}
t^{\rho-1}f(t,\varphi(t))dt. \label{Volterra}
\end{eqnarray}
\begin{theorem}\label{equivalence}
Let $\gamma=\alpha+\beta(1-\alpha)$, where $0<\alpha<1$ and $0\leq{\beta}\leq{1}$. 
If $f:(a,b]\times\mathbb{R}\rightarrow{\mathbb{R}}$ is a function such that 
$f(\cdot,\varphi(\cdot))\in{C}_{1-\gamma}[a,b]$ for any $\varphi\in{C}_{1-\gamma}[a,b]$, 
then $\varphi$ satisfies \textnormal{Eq.(\ref{Cauchy})-Eq.(\ref{cond-inicial})} if 
and only if it satisfies \textnormal{Eq.(\ref{Volterra})}.
\end{theorem}
\begin{proof}
$(\Rightarrow)$ Let $\varphi\in{C}_{1-\gamma}^{\gamma}[a,b]$ be a solution of the problem 
\textnormal{Eq.(\ref{Cauchy})-Eq.(\ref{cond-inicial})}. We prove that $\varphi$ 
is also a solution of \textnormal{Eq.(\ref{Volterra})}. From the definition of 
${C}_{1-\gamma}^{\gamma}[a,b]$, \textnormal{Lemma \ref{Lem-3}} and using 
\textnormal{Definition \ref{Hilfer-Katugampola}}, we have
\begin{eqnarray*}
{^{\rho}\mathcal{J}_{a^+}^{1-\gamma}}\varphi\in{C}[a,b] \quad \textnormal{and} \quad {^{\rho}
\mathcal{D}_{a^+}^{\gamma}}\varphi=\delta_{\rho}\,{^{\rho}\mathcal{J}_{a^+}^{1-\gamma}}
\varphi\in{C}_{1-\gamma}[a,b].
\end{eqnarray*}
By \textnormal{Definition \ref{espacos}}, it follows that 
\begin{eqnarray*}
{^{\rho}\mathcal{J}_{a^+}^{1-\gamma}}\varphi\in{C}_{1-\gamma}^{1}[a,b].
\end{eqnarray*}
Using \textnormal{Lemma \ref{L4}}, with $\alpha=\gamma$, and 
\textnormal{Eq.(\ref{cond-inicial})}, we can write
\begin{eqnarray}
({^{\rho}\mathcal{J}_{a^+}^{\gamma}}\,{^{\rho}\mathcal{D}_{a^+}^{\gamma}}\varphi)(x)=
\varphi(x)-\frac{c}{\Gamma(\gamma)}\left(\frac{x^{\rho}-a^{\rho}}{\rho}\right)^{\gamma-1}, \label{Eq1-T1}
\end{eqnarray}
where $x\in(a,b]$. By hypothesis, ${^{\rho}\mathcal{D}_{a^+}^{\gamma}}\varphi\in{C}_{1-\gamma}[a,b]$, 
using \textnormal{Lemma \ref{L5}} with $\alpha=\gamma$ and \textnormal{Eq.(\ref{Cauchy})}, 
we have
\begin{eqnarray}
({^{\rho}\mathcal{J}_{a^+}^{\gamma}}\,{^{\rho}\mathcal{D}_{a^+}^{\gamma}}\varphi)(x)&=&
({^{\rho}\mathcal{J}_{a^+}^{\alpha}}\,{^{\rho}\mathcal{D}_{a^+}^{\alpha,\beta}}\varphi)(x)\nonumber\\
&=&({^{\rho}\mathcal{J}_{a^+}^{\alpha}}f(t,\varphi(t)))(x). \label{Eq2-T1}
\end{eqnarray}
Comparing \textnormal{Eq.(\ref{Eq1-T1})} and \textnormal{Eq.(\ref{Eq2-T1})}, we see that 
\begin{eqnarray}
\varphi(x)=\frac{c}{\Gamma(\gamma)}\left(\frac{x^{\rho}-a^{\rho}}{\rho}\right)^{\gamma-1}
+({^{\rho}\mathcal{J}_{a^+}^{\alpha}}f(t,\varphi(t)))(x), \label{Eq3-T1}
\end{eqnarray}
with $x\in(a,b]$, that is, $\varphi (x)$ satisfies \textnormal{Eq.(\ref{Volterra})}.\\
\\
$(\Leftarrow)$ Let $\varphi\in{C}_{1-\gamma}^{\gamma}[a,b]$ satisfying 
\textnormal{Eq.(\ref{Volterra})}. We prove that $\varphi$ also satisfies the problem 
\textnormal{Eq.(\ref{Cauchy})-Eq.(\ref{cond-inicial})}. Apply operator 
${^{\rho}\mathcal{D}_{a^+}^{\gamma}}$ on both sides of \textnormal{Eq.(\ref{Eq3-T1})}. 
Then, from \textnormal{Lemma \ref{L2}}, \textnormal{Lemma \ref{L5}} and 
\textnormal{Definition \ref{Hilfer-Katugampola}} we obtain 
\begin{eqnarray}
({^{\rho}\mathcal{D}_{a^+}^{\gamma}}\varphi)(x)=
({^{\rho}\mathcal{D}_{a^+}^{\beta(1-\alpha)}}f(t,\varphi(t)))(x). \label{Eq4-T1}
\end{eqnarray}
By hypothesis, ${^{\rho}\mathcal{D}_{a^+}^{\gamma}}\varphi\in{C}_{1-\gamma}[a,b]$;  
then, \textnormal{Eq.(\ref{Eq4-T1})} implies that 
\begin{eqnarray}
({^{\rho}\mathcal{D}_{a^+}^{\gamma}}\varphi)(x)&=&\delta_{\rho}\,{^{\rho}
\mathcal{J}_{a^+}^{1-\beta(1-\alpha)}}\varphi)(x)\nonumber\\
&=&({^{\rho}\mathcal{D}_{a^+}^{\beta(1-\alpha)}}
\varphi)\in{C}_{1-\gamma}[a,b]. \label{Eq5-T1}
\end{eqnarray}
As $f(\cdot,\varphi(\cdot))\in{C}_{1-\gamma}[a,b]$ and from \textnormal{Lemma \ref{Lem-3}}, 
follows
\begin{eqnarray}
{^{\rho}\mathcal{J}_{a^+}^{1-\beta(1-\alpha)}}f\in{C}_{1-\gamma}[a,b]. \label{Eq6-T1}
\end{eqnarray}
From \textnormal{Eq.(\ref{Eq5-T1})}, \textnormal{Eq.(\ref{Eq6-T1})} and 
\textnormal{Definition \ref{espacos}}, we obtain
\begin{eqnarray*}
{^{\rho}\mathcal{J}_{a^+}^{1-\beta(1-\alpha)}}\varphi\in{C}_{1-\gamma}^{1}[a,b].
\end{eqnarray*}
Applying operator ${^{\rho}\mathcal{J}_{a^+}^{\beta(1-\alpha)}}$ on both sides of 
\textnormal{Eq.(\ref{Eq5-T1})} and using \textnormal{Lemma \ref{L4}}, \textnormal{Lemma \ref{L2}} 
and \textnormal{Definition \ref{Hilfer-Katugampola}}, we have
\begin{eqnarray*}
({^{\rho}\mathcal{J}_{a^+}^{\beta(1-\alpha)}}\,{^{\rho}\mathcal{D}_{a^+}^{\gamma}}
\varphi)(x)&=&f(x,\varphi(x))
+\frac{({^{\rho}\mathcal{J}_{a^+}^{1-\beta(1-\alpha)}}
f(t,\varphi(t)))(a)}{\Gamma(\beta(1-\alpha))}\left(\frac{x^{\rho}-t^{\rho}}{\rho}\right)^{\beta(1-\alpha)-1}\\
&=&({^{\rho}\mathcal{D}_{a^+}^{\alpha,\beta}}\varphi)(x)=f(x,\varphi(x)),
\end{eqnarray*}
that is, \textnormal{Eq.(\ref{Cauchy})} holds. Next, we prove that if 
$\varphi\in{C}_{1-\gamma}^{\gamma}[a,b]$ satisfies \textnormal{Eq.(\ref{Volterra})}, 
it also satisfies \textnormal{Eq.(\ref{cond-inicial})}. To this end, we multiply both sides of 
\textnormal{Eq.(\ref{Eq3-T1})} by ${^{\rho}\mathcal{J}_{a^+}^{1-\gamma}}$ 
and use \textnormal{Lemma \ref{L2}} and \textnormal{Theorem \ref{semigrupo}} to get
\begin{eqnarray}
({^{\rho}\mathcal{J}_{a^+}^{1-\gamma}}\varphi)(x)=c+({^{\rho}\mathcal{J}_{a^+}^{1-\gamma(1-\alpha)}}
f(t,\varphi(t)))(x). \label{Eq7-T1}
\end{eqnarray}
Finally, taking $x\rightarrow{a}$ in \textnormal{Eq.(\ref{Eq7-T1})},
\textnormal{Eq.(\ref{cond-inicial})} follows.
\end{proof}
\section{Existence and uniqueness of solution for the Cauchy problem}

In this section we prove the existence and uniqueness of the solution for the
problem \textnormal{Eq.(\ref{Cauchy})-Eq.(\ref{cond-inicial})} in the space
${C}_{1-\gamma, \rho}^{\alpha,\beta}[a,b]$ defined in \textnormal{Property
\ref{espacos2}}, under the hypotheses of \textnormal{Theorem \ref{equivalence}}
and the Lipschitz condition on $f(\cdot,\varphi)$ \cite[p.136]{kilbas} with
respect to the second variable, that is, that $f(\cdot,\varphi)$ is bounded in a
region $G\subset\mathbb{R}$ such that

\begin{eqnarray}
{\lVert f(x,\varphi_{1})-f(x,\varphi_{1})\rVert}_{{C}_{1-\gamma,\rho}[a,b]}
\leq{A}{\lVert \varphi_{1}-\varphi_{2} \rVert}_{{C}_{1-\gamma,\rho}[a,b]}, \label{Lipschitz}
\end{eqnarray}
for all $x\in(a,b]$, and for all $\varphi_1,\varphi_2\in{G}$, where $A>0$ is constant.
\begin{theorem}\label{exis&uniq}
Let $0<\alpha<1$, $0\leq{\beta}\leq{1}$ and $\gamma=\alpha+\beta(1-\alpha)$. 
Let $f:(a,b]\times\mathbb{R}\rightarrow\mathbb{R}$ be a function such that  
$f(\cdot,\varphi(\cdot))\in{C}_{\mu,\rho}[a,b]$ for any  $\varphi\in{C}_{\mu,\rho}[a,b]$ 
with $1-\gamma\leq\mu<{1-\beta(1-\alpha)}$ and satisfying the Lipschitz condition 
\textnormal{Eq.(\ref{Lipschitz})} with respect to the second variable. Then, there exists
a unique solution $\varphi$ for the problem \textnormal{Eq.(\ref{Cauchy})-Eq.(\ref{cond-inicial})} 
in the space ${C}_{1-\gamma, \mu}^{\alpha,\beta}[a,b]$.
\begin{proof}
According to \textnormal{Theorem \ref{equivalence}}, we just have to prove that there exists
	a unique solution for the Volterra integral equation,
	\textnormal{Eq.(\ref{Volterra})}. This equation can be written as
\begin{eqnarray*}
\varphi(x)=T\varphi(x),
\end{eqnarray*}
where
\begin{eqnarray}
T\varphi(x)=\varphi_{0}(x)+[{^{\rho}\mathcal{J}_{a^+}^{\alpha}}f(t,\varphi(t))](x), \label{operator-T}
\end{eqnarray}
with
\begin{eqnarray}
\varphi_{0}(x)=\frac{c}{\Gamma(\gamma)}\left(\frac{x^{\rho}-a^{\rho}}{\rho}\right)^{\gamma-1}. \label{cond}
\end{eqnarray}


Thus, we divide the interval $(a,b]$ into
subintervals on which operator $T$ is a contraction; we then use Banach fixed point theorem,  
\textnormal{Theorem \ref{fixed-point}}.
Note that $\varphi\in{C}_{1-\gamma,\rho}[a,x_1]$, 
where $a=x_0<{x_1}<\ldots<{x_M}={b}$
and${C}_{1-\gamma,\rho}[a,x_1]$ is a complete metric space with metric 
\begin{eqnarray*}
d(\varphi_{1},\varphi_{2})={\lVert \varphi_{1}-\varphi_{2}\rVert}_{{C}_{1-\gamma,\rho}[a,x_1]}
=\max_{x\in[a,x_1]}{\bigg|\left(\frac{x^{\rho}-a^{\rho}}{\rho}\right)^{1-\gamma}
[\varphi_{1}-\varphi_{2}]\bigg|}.
\end{eqnarray*}
Choose $x_{1}\in(a,b]$ such that the inequality
\begin{eqnarray}
w_{1}=\frac{A\,\Gamma(\gamma)}{\Gamma(\alpha+\gamma)}
\left(\frac{{x_{1}}^{\rho}-a^{\rho}}{\rho}\right)^{\alpha}<1, \label{hypothesis}
\end{eqnarray}
where $A>0$ is a constant, holds, as in \textnormal{Eq.(\ref{Lipschitz})}.
Thus, $\varphi_{0}\in{C}_{1-\gamma,\rho}[a,x_1]$ and from \textnormal{Lemma \ref{Lem-3}}, 
we have $T\varphi\in{C}_{1-\gamma,\rho}[a,x_1]$ and $T$ maps ${C}_{1-\gamma,\rho}[a,x_1]$ 
into ${C}_{1-\gamma,\rho}[a,x_1]$. Therefore, from \textnormal{Eq.(\ref{Lipschitz})}, 
\textnormal{Eq.(\ref{operator-T})}, \textnormal{Lemma \ref{Lem-3}} and for any 
	$\varphi_{1},\varphi_{2}\in{C}_{1-\gamma,\rho}[a,x_1]$, we can write
\begin{eqnarray*}
{\lVert T\varphi_{1}-T\varphi_{2} \rVert}_{{C}_{1-\gamma}[a,x_1]}&=&{\lVert {^{\rho}
\mathcal{J}_{a^+}^{\alpha}}f(t,\varphi_{1}(t))-\,{^{\rho}\mathcal{J}_{a^+}^{\alpha}}
f(t,\varphi_{2}(t)) \rVert}_{{C}_{1-\gamma,\rho}[a,x_1]}\\
&=&{\lVert {^{\rho}\mathcal{J}_{a^+}^{\alpha}}[|f(t,\varphi_{1}(t))-f(t,\varphi_{2}(t))|] \rVert}_
{{C}_{1-\gamma,\rho}[a,x_1]}\\
&\leq&\left(\frac{{x_{1}}^{\rho}-a^{\rho}}{\rho}\right)^{\alpha}\frac{\Gamma(\gamma)}
{\Gamma(\alpha+\gamma)}\\
&\times&{\lVert f(t,\varphi_{1}(t))-f(t,\varphi_{2}(t)) \rVert}_
{{C}_{1-\gamma,\rho}[a,x_1]}\\
&\leq& A\left(\frac{{x_{1}}^{\rho}-a^{\rho}}{\rho}\right)^{\alpha}\frac{\Gamma(\gamma)}
{\Gamma(\alpha+\gamma)}{\lVert \varphi_{1}(t)-\varphi_{2}(t) \rVert}_{{C}_{1-\gamma,\rho}[a,x_1]}\\
&\leq&{w_1}{\lVert \varphi_{1}(t)-\varphi_{2}(t) \rVert}_{{C}_{1-\gamma,\rho}[a,x_1]}.
\end{eqnarray*}
By hypothesis \textnormal{Eq.(\ref{hypothesis})} we can use the Banach fixed point to get
a unique solution $\varphi^{*}\in{{C}_{1-\gamma,\rho}[a,x_1]}$ for \textnormal{Eq.(\ref{Volterra})} 
on the interval $(a,x_1]$. This solution $\varphi^{*}$ is obtained as a limit of a convergent sequence 
$T^{k}{\varphi}_{0}^{*}$:
\begin{eqnarray}
\lim_{k \to \infty}{\lVert T^{k}{\varphi}_{0}^{*}-{\varphi}^{*} \rVert}_{{C}_{1-\gamma,\rho}[a,x_1]}=0, \label{converg}
\end{eqnarray}
where ${\varphi}_{0}^{*}$ is any function in ${C}_{1-\gamma,\rho}[a,x_1]$ and
\begin{eqnarray*}
(T^{k}{\varphi}_{0}^{*})(x)=(T\,T^{k-1}{\varphi}_{0}^{*})(x)
={\varphi}_{0}(x)+[{^{\rho}\mathcal{J}_{a^+}^{\alpha}}
f(t,({T^{k-1}}{\varphi}_{0}^{*})(t))](x), \quad k\in\mathbb{N}.
\end{eqnarray*}
We take ${\varphi}_{0}^{*}(x)={\varphi}_{0}(x)$ with ${\varphi}_{0}(x)$ defined by
\textnormal{Eq.(\ref{cond})}. Denoting
\begin{eqnarray}
\varphi_{k}(x)=(T^{k}\varphi_{0}^{*})(x), \quad k\in\mathbb{N}, \label{Eq.Fix}
\end{eqnarray}
then \textnormal{Eq.(\ref{Eq.Fix})} admits the form 
\begin{eqnarray*}
\varphi_{k}(x)={\varphi}_{0}(x)+[{^{\rho}\mathcal{J}_{a^+}^{\alpha}}f(t,{\varphi_{k-1}}(t))](x), \quad k\in\mathbb{N}.
\end{eqnarray*}
On the other hand, \textnormal{Eq.(\ref{converg})} can be rewritten as
\begin{eqnarray*}
\lim_{k \to \infty}{\lVert {\varphi}_{k}-{\varphi}^{*} \rVert}_{{C}_{1-\gamma,\rho}[a,x_1]}=0.
\end{eqnarray*}
We consider the interval $[x_1,x_2]$, where $x_2=x_1+h_1$, $h_1>0$ and $x_2<b$, 
then by \textnormal{Eq.(\ref{Volterra})}, we can write
\begin{eqnarray*}
\varphi(x)&=&\frac{c}{\Gamma(\gamma)}\left(\frac{x^{\rho}-a^{\rho}}{\rho}\right)^{\gamma-1}\\
&+&\frac{1}{\Gamma(\alpha)}\int_{a}^{x_1}t^{\rho-1}\left(\frac{x^{\rho}-t^{\rho}}{\rho}\right)^{\alpha-1}
f(t,\varphi(t))dt\\
&+&\frac{1}{\Gamma(\alpha)}\int_{x_1}^{x}t^{\rho-1}\left(\frac{x^{\rho}-t^{\rho}}{\rho}\right)^{\alpha-1}
f(t,\varphi(t))dt\\
&=&\varphi_{01}(x)+\frac{1}{\Gamma(\alpha)}\int_{x_1}^{x}t^{\rho-1}\left(\frac{x^{\rho}-t^{\rho}}
{\rho}\right)^{\alpha-1}f(t,\varphi(t))dt,
\end{eqnarray*}
where $\varphi_{01}(x)$,  defined by
\begin{eqnarray*}
\varphi_{01}(x)&=&\frac{c}{\Gamma(\gamma)}\left(\frac{x^{\rho}-a^{\rho}}{\rho}\right)^{\gamma-1}
+\frac{1}{\Gamma(\alpha)}\int_{a}^{x_1}t^{\rho-1}\left(\frac{x^{\rho}-t^{\rho}}{\rho}\right)^{\alpha-1}
f(t,\varphi(t))dt,
\end{eqnarray*}
is the known function and $\varphi_{01}(x)\in{{C}_{1-\gamma,\rho}[x_1,x_2]}$. Using the same
arguments as above, we conclude that there exists a unique solution $\varphi^{*}\in{C}_{1-\gamma,\rho}[x_1,x_2]$ 
for \textnormal{Eq.(\ref{Volterra})} on the interval $[x_1,x_2]$. The next interval to be
considered is $[x_2,x_3]$, where $x_3=x_2+h_2$, $h_2>0$ and $x_3<b$. Repeating this process,
we conclude that there exists a unique solution $\varphi^{*}\in{C}_{1-\gamma,\rho}[a,b]$ for
\textnormal{Eq.(\ref{Volterra})} on the interval $[a,b]$.
We must show that such unique solution $\varphi^{*}\in{C}_{1-\gamma,\rho}[a,b]$ is also in 
${C}_{1-\gamma,\mu}^{\alpha,\beta}[a,b]$. Thus, we need show that 
$({^{\rho}\mathcal{D}^{\alpha,\beta}_{a^+}}\varphi^{*})\in{C}_{\mu,\rho}[a,b]$.
We emphasize that $\varphi^{*}$ is the limit of the sequence $\varphi_{k}$, where 
$\varphi_{k}=T^{k}{\varphi_{0}^{*}}\in{C}_{1-\gamma,\rho}[a,b]$, that is,
\begin{eqnarray}
\lim_{k \to \infty}{\lVert \varphi_{k}-\varphi^{*} \rVert}_{C_{1-\gamma,\rho}[a,b]}=0, \label{norm}
\end{eqnarray}
for an adequate choice of $\varphi_{0}^{*}(x)$ on each subinterval $[a,x_1]$,
\ldots,$[x_{M-1},b]$. If $\varphi_{0}(x)\neq{0}$, then we can admit $\varphi_{0}^{*}(x)=\varphi_{0}(x)$
and once $\mu\geq{1-\gamma}$, from Lipschitz condition, \textnormal{Eq.(\ref{Lipschitz})}, 
and by \textnormal{Lemma \ref{L0}}, we can write
\begin{footnotesize}
\begin{eqnarray}
{\lVert {^{\rho}\mathcal{D}^{\alpha,\beta}_{a^+}}\varphi_{k} -\, 
{^{\rho}\mathcal{D}^{\alpha,\beta}_{a^+}}\varphi^{*} \rVert}_{C_{\mu,\rho}[a,b]}
&=&{\lVert f(x,\varphi_{k})-f(x,\varphi^{*}) \rVert}_{C_{\mu,\rho}[a,b]}\nonumber\\
&\leq&A\left(\frac{b^{\rho}-a^{\rho}}{\rho}\right)^{\mu-1+\gamma}
{\lVert \varphi_{k}-\varphi^{*} \rVert}_{C_{1-\gamma,\rho}[a,b]}. \label{norm-der}
\end{eqnarray}
\end{footnotesize}
By \textnormal{Eq.(\ref{norm})} and \textnormal{Eq.(\ref{norm-der})}, we obtain
\begin{eqnarray*}
\lim_{k \to \infty}{\lVert {^{\rho}\mathcal{D}^{\alpha,\beta}_{a^+}}
\varphi_{k}-\, {^{\rho}\mathcal{D}^{\alpha,\beta}_{a^+}}\varphi^{*} \rVert}_{C_{\mu,\rho}[a,b]}=0.
\end{eqnarray*}
From this last expression, we have
$({^{\rho}\mathcal{D}^{\alpha,\beta}_{a^+}}\varphi^{*}) \in{C}_{\mu,\rho}[a,b]$
if
$({^{\rho}\mathcal{D}^{\alpha,\beta}_{a^+}}\varphi_{k})\in{C}_{\mu,\rho}[a,b]$,
$k=1,2,\ldots$ Since
$({^{\rho}\mathcal{D}^{\alpha,\beta}_{a^+}}\varphi_{k})(x)=f(x,\varphi_{k-1}(x))$,
then by the previous argument, we obtain that
$f(\cdot,\varphi^{*}(\cdot))\in{C}_{\mu,\rho}[a,b]$ for any
$\varphi^{*}\in{C}_{\mu,\rho}[a,b]$. Consequently,
$\varphi^{*}\in{C}_{1-\gamma,\mu}^{\alpha,\beta}[a,b]$.
\end{proof}
\end{theorem}



\section{Cauchy-type problems for fractional differential equations}

In this section, we present explicit solutions to fractional differential
equations involving the Hilfer-Katugampola differential operator
$({^{\rho}\mathcal{D}^{\alpha,\beta}_{a^+}}\varphi)(x)$ of order $0<\alpha<1$
and type $0\leq{\beta}\leq{1}$ in the space
${C}_{1-\gamma,\rho}^{\alpha,\beta}[a,b]$ defined in \textnormal{Property
\ref{espacos2}}.

We consider the following Cauchy problem
\begin{eqnarray}
&&({^{\rho}\mathcal{D}^{\alpha,\beta}_{a^+}}\varphi)(x)-\lambda\varphi(x)=f(x), \quad 0<\alpha<1,\quad 0\leq{\beta}\leq{1}, \label{Cauchy1}\\
&&(^{\rho}\mathcal{J}_{a^+}^{1-\gamma})(a)=c, \quad \gamma=\alpha+\beta(1-\alpha), \label{Cauchy2}
\end{eqnarray}
where $c,\lambda\in\mathbb{R}$.
We suppose that $f(x)\in{{C}_{\mu,\rho}[a,b]}$ with $0\leq{\mu}<1$ and $\rho>0$. 
Then, by \textnormal{Theorem \ref{equivalence}}, the problem \textnormal{Eq.(\ref{Cauchy1})-Eq.(\ref{Cauchy2})} 
is equivalent to solve the following integral equation
\begin{eqnarray}
\varphi(x)&=&\frac{c}{\Gamma(\gamma)}\left(\frac{x^{\rho}-a^{\rho}}{\rho}\right)^{\gamma-1}
+\frac{\lambda}{\Gamma(\alpha)}\int_{a}^{x}t^{\rho-1}\left(\frac{x^{\rho}-t^{\rho}}{\rho}\right)^{\alpha-1}\varphi(t)dt \nonumber\\
&+&\frac{1}{\Gamma(\alpha)}\int_{a}^{x}t^{\rho-1}\left(\frac{x^{\rho}-t^{\rho}}{\rho}\right)^{\alpha-1}f(t)dt. \label{Volt}
\end{eqnarray}
In order to solve \textnormal{Eq.(\ref{Volt})}, we use the method of successive approximations,
that is,
\begin{eqnarray}
\varphi_{0}(x)=\frac{c}{\Gamma(\gamma)}\left(\frac{x^{\rho}-a^{\rho}}{\rho}\right)^{\gamma-1}, \label{phi0}
\end{eqnarray}
\begin{eqnarray}
\varphi_{k}(x)&=&\varphi_{0}(x)+\frac{\lambda}{\Gamma(\alpha)}\int_{a}^{x}t^{\rho-1}
\left(\frac{x^{\rho}-t^{\rho}}{\rho}\right)^{\alpha-1}\varphi_{k-1}(t)dt \nonumber\\
&+&\frac{1}{\Gamma(\alpha)}\int_{a}^{x}t^{\rho-1}\left(\frac{x^{\rho}-t^{\rho}}
{\rho}\right)^{\alpha-1}f(t)dt, \quad (k\in\mathbb{N}). \label{phi-k}
\end{eqnarray}
Using \textnormal{Eq.(\ref{int-gen-esq})},  \textnormal{Eq.(\ref{phi0})} and 
\textnormal{Lemma \ref{L2}}, we have the following expression for $\varphi_{1}(x)$: 
\begin{eqnarray}
\varphi_{1}(x)&=&\varphi_{0}(x)+({^{\rho}\mathcal{J}_{a^+}^{\alpha}}\varphi_{0})(x)
+({^{\rho}\mathcal{J}_{a^+}^{\alpha}}f)(x) \nonumber\\
&=&c\sum_{j=1}^{2}\frac{\lambda^{j-1}}{\Gamma(\alpha{j}+\beta(1-\alpha))}
\left(\frac{x^{\rho}-a^{\rho}}{\rho}\right)^{\alpha{j}+\beta(1-\alpha)-1}+
({^{\rho}\mathcal{J}_{a^+}^{\alpha}}f)(x). \label{phi-1}
\end{eqnarray}
Similarly, using \textnormal{Eq.(\ref{phi0})}, \textnormal{Eq.(\ref{phi-k})},
\textnormal{Eq.(\ref{phi-1})} and \textnormal{Property \ref{semigrupo}}, we get  
an expression for $\varphi_{2}(x)$, as follows: 
\begin{eqnarray*}
\varphi_{2}(x)&=&\varphi_{0}(x)+({^{\rho}\mathcal{J}_{a^+}^{\alpha}}\varphi_{1})(x)
+({^{\rho}\mathcal{J}_{a^+}^{\alpha}}f)(x) \nonumber\\
&=&c\sum_{j=1}^{3}\frac{\lambda^{j-1}}{\Gamma(\alpha{j}+\beta(1-\alpha))}
\left(\frac{x^{\rho}-a^{\rho}}{\rho}\right)^{\alpha{j}+\beta(1-\alpha)-1}\\
&+&\lambda({^{\rho}\mathcal{J}_{a^+}^{\alpha}}{^{\rho}\mathcal{J}_{a^+}^{\alpha}}f)(x)+
({^{\rho}\mathcal{J}_{a^+}^{\alpha}}f)(x) \label{phi-2} \\
&=&c\sum_{j=1}^{3}\frac{\lambda^{j-1}}{\Gamma(\alpha{j}+\beta(1-\alpha))}
\left(\frac{x^{\rho}-a^{\rho}}{\rho}\right)^{\alpha{j}+\beta(1-\alpha)-1}\\
&+&\int_{a}^{x}\sum_{j=1}^{2}\frac{\lambda^{j-1}}{\Gamma(\alpha{j})}\,t^{\rho-1}
\left(\frac{x^{\rho}-t^{\rho}}{\rho}\right)^{\alpha{j}-1}f(t)dt.
\end{eqnarray*}
Continuing this process, the expression for $\varphi_{k}(x)$ is given by
\begin{eqnarray*}
\varphi_{k}(x)&=&c\sum_{j=1}^{k+1}\frac{\lambda^{j-1}}{\Gamma(\alpha{j}+
\beta(1-\alpha))}\left(\frac{x^{\rho}-a^{\rho}}{\rho}\right)^{\alpha{j}+
\beta(1-\alpha)-1}\\
&+&\int_{a}^{x}\sum_{j=1}^{k}\frac{\lambda^{j-1}}
{\Gamma(\alpha{j})}\,t^{\rho-1}\left(\frac{x^{\rho}-t^{\rho}}{\rho}\right)^{\alpha{j}-1}f(t)dt.
\end{eqnarray*}
Taking the limit $k\rightarrow{\infty}$, we obtain the expression for $\varphi(x)$, that is,
\begin{eqnarray*}
\varphi(x)&=&c\sum_{j=1}^{\infty}\frac{\lambda^{j-1}}{\Gamma(\alpha{j}+\beta(1-\alpha))}
\left(\frac{x^{\rho}-a^{\rho}}{\rho}\right)^{\alpha{j}+\beta(1-\alpha)-1}\\
&+&\int_{a}^{x}\sum_{j=1}^{\infty}\frac{\lambda^{j-1}}{\Gamma(\alpha{j})}\,
t^{\rho-1}\left(\frac{x^{\rho}-t^{\rho}}{\rho}\right)^{\alpha{j}-1}f(t)dt.
\end{eqnarray*}
Changing the summation index in this last expression, $j\rightarrow{j+1}$, we have
\begin{eqnarray*}
\varphi(x)&=&c\sum_{j=0}^{\infty}\frac{\lambda^{j}}{\Gamma(\alpha{j}+\gamma)}
\left(\frac{x^{\rho}-a^{\rho}}{\rho}\right)^{\alpha{j}+\gamma-1}\\
&+&\int_{a}^{x}
\sum_{j=0}^{\infty}\frac{\lambda^{j}}{\Gamma(\alpha{j}+\alpha)}\,t^{\rho-1}
\left(\frac{x^{\rho}-t^{\rho}}{\rho}\right)^{\alpha{j}+\alpha-1}f(t)dt . 
\end{eqnarray*}
Using \textnormal{Definition \ref{Mittag}}, we can rewrite the solution in terms of 
two-parameters Mittag-Leffler functions,
\begin{eqnarray}
\varphi(x)&=&c\left(\frac{x^{\rho}-a^{\rho}}{\rho}\right)^{\gamma-1}E_{\alpha,\gamma}
\left[\lambda\left(\frac{x^{\rho}-a^{\rho}}{\rho}\right)^{\alpha}\right]\nonumber\\
&+&\int_{a}^{x}t^{\rho-1}\left(\frac{x^{\rho}-t^{\rho}}{\rho}\right)^{\alpha-1}
E_{\alpha,\alpha}\left[\lambda\left(\frac{x^{\rho}-t^{\rho}}{\rho}\right)^{\alpha}\right]
f(t)dt. \label{sol-Cauchy}
\end{eqnarray}
The function $f(x,\varphi)=\lambda{\varphi(x)}+f(x)$ satisfies the Lipschitz condition, 
\textnormal{Eq.(\ref{Lipschitz})}, for any $x_1,x_2\in(a,b]$ and any $y\in{G}$, where $G$ 
is an open set on $\mathbb{R}$. If $\mu\geq{1-\gamma}$, then by \textnormal{Theorem \ref{exis&uniq}}, 
the problem \textnormal{Eq.(\ref{Cauchy1})-Eq.(\ref{Cauchy2})} has a unique solution
given by \textnormal{Eq.(\ref{sol-Cauchy})} in the space ${C}_{1-\gamma,\mu}^{\alpha,\beta}[a,b]$.
Note that the problem \textnormal{Eq.(\ref{Cauchy1})-Eq.(\ref{Cauchy2})}, whose solution
is given by \textnormal{Eq.(\ref{sol-Cauchy})}, includes the following particular cases:

\begin{itemize}

\item If $\rho\rightarrow{1}$ and $\beta=0$, then $\gamma=\alpha$ and we
have a problem involving the Riemann-Liouville fractional derivative; its
solution can be found in \cite[p.224]{kilbas}.

\item For $\rho\rightarrow{1}$ and $\beta=1$ our derivative becomes 
the Caputo fractional derivative; the solution can be found in \cite[p.231]{kilbas}.

\item Considering $\rho\rightarrow{0^+}$ and $\beta=0$, we have $\gamma=\alpha$ 
we have a Cauchy problem formulated with 
the Hadamard fractional derivative; the solution can be found in \cite[p.235]{kilbas}.

\item Other particular cases arise when we vary the parameters as described in 
\textnormal{Property \ref{Hilfer-Katugampola-parameters}}. 
Some of them are presented below.

\end{itemize}
A special case is the case when $f(x)=0$; we then have the following problem
\begin{eqnarray}
&&({^{\rho}\mathcal{D}^{\alpha,\beta}_{a^+}}\varphi)(x)-\lambda\varphi(x)=0, 
\quad 0<\alpha<1,\quad 0\leq{\beta}\leq{1},\\
&&(^{\rho}\mathcal{J}_{a^+}^{1-\gamma})(a)=c, \quad \gamma=\alpha+\beta(1-\alpha),
\end{eqnarray}
with $\lambda\in\mathbb{R}$ and $a<{x}\leq{b}$. The solution is given by
\begin{eqnarray}
\varphi(x)=c\left(\frac{x^{\rho}-a^{\rho}}{\rho}\right)^{\gamma-1}E_{\alpha,\gamma}
\left[\lambda\left(\frac{x^{\rho}-a^{\rho}}{\rho}\right)^{\alpha}\right].
\end{eqnarray}
Also, consider the following Cauchy problem:
\begin{eqnarray}
&&({^{\rho}\mathcal{D}^{\alpha,\beta}_{a^+}}\varphi)(x)-\lambda\left(\frac{x^{\rho}-a^{\rho}}{\rho}\right)
^{\xi}\varphi(x)=0, \; 0<\alpha<1,\; 0\leq{\beta}\leq{1},  \label{p-Cauchy1}\\
&&(^{\rho}\mathcal{J}_{a^+}^{1-\alpha})(a)=c, \quad c\in\mathbb{R}, \quad \rho>0 
, \label{p-Cauchy2}
\end{eqnarray}
with $\lambda,\xi\in\mathbb{R}$, $a<{x}\leq{b}$ and $\xi>-\alpha$. We suppose 
$\displaystyle \left[\lambda\left(\frac{x^{\rho}-a^{\rho}}{\rho}\right)^{\xi}\varphi\right]\in{{C}_{1-\alpha,\rho}[a,b]}$. 
Then, by \textnormal{Theorem \ref{equivalence}}, the problem \textnormal{Eq.(\ref{p-Cauchy1})-Eq.(\ref{p-Cauchy2})} 
is equivalent to the following integral equation: 
\begin{eqnarray}
\varphi(x)&=&\frac{c}{\Gamma(\alpha)}\left(\frac{x^{\rho}-a^{\rho}}{\rho}\right)^{\alpha-1}
+\frac{\lambda}{\Gamma(\alpha)}\int_{a}^{x}t^{\rho-1}\left(\frac{x^{\rho}-t^{\rho}}{\rho}\right)^{\alpha-1}
\left(\frac{t^{\rho}-a^{\rho}}{\rho}\right)^{\xi}\varphi(t)dt. \label{p-Volt}
\end{eqnarray}
We apply the method of successive approximations to solve the integral equation 
\textnormal{Eq.(\ref{p-Volt})}, that is, we consider
\begin{eqnarray}
\varphi_{0}(x)=\frac{c}{\Gamma(\alpha)}\left(\frac{x^{\rho}-a^{\rho}}{\rho}\right)^{\alpha-1}
\end{eqnarray}
and
\begin{eqnarray}
\varphi_{k}(x)&=&\varphi_{0}(x)
+\frac{\lambda}{\Gamma(\alpha)}\int_{a}^{x}t^{\rho-1}
\left(\frac{x^{\rho}-t^{\rho}}{\rho}\right)^{\alpha-1}\left(\frac{t^{\rho}-a^{\rho}}{\rho}\right)^{\xi}\varphi_{k-1}(t)dt.
\end{eqnarray}
For $k=1$ and using \textnormal{Lemma \ref{L2}}, we have
\begin{eqnarray}
\varphi_{1}(x)&=&\varphi_{0}(x)+\lambda\left({^{\rho}\mathcal{J}_{a^+}^{\alpha}}
\left(\frac{t^{\rho}-a^{\rho}}{\rho}\right)^{\xi}\varphi_{0}\right)(x)\nonumber\\
&=&\frac{c}{\Gamma(\alpha)}\left(\frac{x^{\rho}-a^{\rho}}{\rho}\right)^{\alpha-1}+\frac{c\lambda}
{\Gamma(\alpha)}\frac{\Gamma(\alpha+\xi)}{\Gamma(2\alpha+\xi)}
\left(\frac{x^{\rho}-a^{\rho}}{\rho}\right)^{2\alpha+\xi-1}.
\end{eqnarray}
For $k=2$ and using again \textnormal{Lemma \ref{L2}}, we can write
\begin{eqnarray*}
\varphi_{2}(x)&=&\varphi_{0}(x)+\lambda\left({^{\rho}\mathcal{J}_{a^+}^{\alpha}}
\left(\frac{t^{\rho}-a^{\rho}}{\rho}\right)^{\xi}\varphi_{1}\right)(x)\nonumber\\
&=&\varphi_{0}(x)+\frac{c\lambda}{\Gamma(\alpha)}\left({^{\rho}\mathcal{J}_{a^+}^{\alpha}}
\left(\frac{t^{\rho}-a^{\rho}}{\rho}\right)^{\alpha+\xi-1}\right)(x)\nonumber\\
&+&\frac{c\lambda^2}{\Gamma(\alpha)}\frac{\Gamma(\alpha+\xi)}{\Gamma(2\alpha+\xi)}\left({^{\rho}
\mathcal{J}_{a^+}^{\alpha}}\left(\frac{t^{\rho}-a^{\rho}}{\rho}\right)^{2\alpha+2\xi-1}\right)(x)\nonumber\\
&=&\frac{c}{\Gamma(\alpha)}\left(\frac{x^{\rho}-a^{\rho}}{\rho}\right)^{\alpha-1}
\left\{1+{c_1}\left[\lambda\left(\frac{x^{\rho}-a^{\rho}}{\rho}\right)^{\alpha+\xi}\right]\right.\\
&+&\left.{c_2}\left[\lambda\left(\frac{x^{\rho}-a^{\rho}}{\rho}\right)^{\alpha+\xi}\right]^{2}\right\},
\end{eqnarray*}
where
\begin{eqnarray}
c_1=\frac{\Gamma(\alpha+\xi)}{\Gamma(2\alpha+\xi)} \quad \mbox{and} \quad 
c_2=\frac{\Gamma(\alpha+\xi)}{\Gamma(2\alpha+\xi)}\frac{\Gamma(2\alpha+2\xi)}{\Gamma(3\alpha+2\xi)}.
\end{eqnarray}
Continuing this process, we obtain the expression for $\varphi_{k}(x)$, given by
\begin{eqnarray}
\varphi_{k}(x)&=&\frac{c}{\Gamma(\alpha)}\left(\frac{x^{\rho}-a^{\rho}}{\rho}\right)^{\alpha-1}
\left\{1+\sum_{j=1}^{k}c_{j}\left[\lambda\left(\frac{x^{\rho}-a^{\rho}}{\rho}\right)^{\alpha+\xi}\right]^{j}\right\},
\label{solution}
\end{eqnarray}
where
\begin{eqnarray}
c_j=\prod_{r=1}^{j}\frac{\Gamma[r(\alpha+\xi)]}{\Gamma[r(\alpha+\xi)+\alpha]}, \quad j\in\mathbb{N}.
\end{eqnarray}
Using \textnormal{Definition \ref{Mittag-generalized}} we can write the solution,
\textnormal{Eq.(\ref{solution})}, in terms of a generalized Mittag-Leffler function:
\begin{eqnarray}
\varphi_{k}(x)&=&\frac{c}{\Gamma(\alpha)}\left(\frac{x^{\rho}-a^{\rho}}{\rho}\right)^{\alpha-1}
\times E_{\alpha,1+\xi/\alpha,1+(\xi-1)/\alpha}\left[\lambda\left(\frac{x^{\rho}-a^{\rho}}{\rho}\right)^{\alpha+\xi}\right].
\label{final-solution}
\end{eqnarray}
If $\xi\geq{0}$, then $\displaystyle
f(x,\varphi)=\lambda\left(\frac{x^{\rho}-a^{\rho}}{\rho}\right)^{\xi}\varphi(x)$
satisfies the Lipschitz condition, \textnormal{Eq.(\ref{Lipschitz})}, for any
$x_1,x_2\in(a,b]$ and for all $\varphi_1,\varphi_2\in{G}$, where $G$ is an
open set on $\mathbb{R}$. If $\mu\geq{1-\gamma}$, then by \textnormal{Theorem
\ref{exis&uniq}}, there exists a unique solution to the problem
\textnormal{Eq.(\ref{p-Cauchy1})-Eq.(\ref{p-Cauchy2})}, given by
\textnormal{Eq.(\ref{final-solution})}, in space
${C}_{1-\gamma,\mu}^{\alpha,\beta}[a,b]$.
Note that the problem \textnormal{Eq.(\ref{p-Cauchy1})-Eq.(\ref{p-Cauchy2})}, whose solution
is given by \textnormal{Eq.(\ref{final-solution})}, admits the following particular cases:

\begin{itemize}

\item For $\rho\rightarrow{1}$ and $\beta=0$, we have formulation for this
problem, as well as it solution considering the Riemann-Liouville
fractional derivative which can be found in
\cite[p.227]{kilbas}.

\item Consider $\rho\rightarrow{1}$ and $\beta=1$, we have formulation of the
problem and it solution, considering the Caputo fractional derivative,
which can be found in \cite[p.233]{kilbas}.

\item For $\rho\rightarrow{0^+}$ and $\beta=0$, we have  formulation for this
Cauchy problem and its solution considering the Hadamard fractional
derivative that can be found in \cite[p.237]{kilbas}.

\item According to the parameters presented in \textnormal{Property
\ref{Hilfer-Katugampola-parameters}} it is also possible to obtain
other particular cases. 

\end{itemize}

We have presented a new fractional derivative, the 
Hilfer-Katugampola fractional derivative. This formulation admits as particular
cases the well-known fractional derivatives of Hilfer, Hilfer-Hadamard,
Riemann-Liouville, Hadamard, Caputo, Caputo-Hadamard, generalized, 
Caputo-type, Weyl and Liouville. 

The equivalence between a nonlinear initial value problem and a Volterra
integral equation was proved. We also discussed the existence and uniqueness of
the solution for this initial value problem. Finally, we obtained the
analytical solutions, using the method of successive approximations, to some
fractional differential equations; some particular cases were recovered.

A possible continuation of this paper consists in defining a fractional integral whose 
kernel contains a Mittag-Leffler function and, using such definition, solving an initial
value problem similar to the one discussed in this paper. This will be published in a
future work \cite{daniela}.

\section*{Acknowledgment} 

We would like to thank the anonymous referees and Dr. J. Em\'{\i}lio Maiorino
for several suggestions and comments that helped improve the paper.



\end{document}